\documentclass[11pt]{article}
\usepackage{amsmath,amssymb,amsfonts,amsthm,graphicx,float}
\usepackage{hyperref}
\newtheorem{theorem}{Theorem}[section]

\newtheorem{lemma}[theorem]{Lemma}
\newtheorem{remark}[theorem]{Remark}

\numberwithin{equation}{section}
\title{A Gauss-Kuzmin-L\'evy theorem for R\'enyi-type continued fractions}
\author{
    Dan Lascu\footnote{e-mail: lascudan@gmail.com.}\nonumber \\
    \emph{\small Mircea cel Batran Naval Academy, 1 Fulgerului, 900218 Constanta,
    Romania} \\
    and\\
    Gabriela Ileana Sebe\footnote{e-mail: igsebe@yahoo.com.} \\
    \emph{\small Politehnica University of Bucharest, Faculty of Applied Sciences},\\
    \emph{\small Splaiul Independentei 313, 060042, Bucharest, Romania} and \\
    \emph{\small Institute of Mathematical Statistics and Applied Mathematics}, \\
     \emph{\small Calea 13 Sept. 13, 050711 Bucharest, Romania}
    }
\begin{document}
\maketitle
\thispagestyle{empty}
\begin{abstract}
We consider an interval map which is a generalization of the R\'enyi transformation.
For the continued fraction expansion arising from this transformation, we prove a result concerning the asymptotic behavior of the distribution functions of this map. More exactly, we use Sz\"usz's method to prove a Gauss-Kuzmin-L\'evy-type theorem.
\end{abstract}
{\bf Mathematics Subject Classifications (2010): 11J70, 60A10}  \\
{\bf Key words}: R\'enyi continued fractions, Gauss-Kuzmin-L\'evy problem

\section{Introduction}
The present paper arises out of series of papers dedicated to R\'enyi-type continued fraction expansions \cite{Sebe&Lascu-2018-1, Sebe&Lascu-2018-2}. Actually, these continued fractions are a particular case of \textit{$u$-backward continued fractions} studied by Gr\"ochenig and Haas \cite{Grochenig&Haas1-1996}.
In 1953, R\'enyi \cite{Renyi-1957} showed that every irrational number $x \in [0, 1)$ has an infinite continued fraction expansion of the form
\begin{equation}
x = 1 - \displaystyle \frac{1}{n_1 - \displaystyle \frac{1}{n_2 - \displaystyle \frac{1}{n_3 - \ddots }}} =: [n_1, n_2, n_3, \ldots]_b, \label{1.1}
\end{equation}
where each $n_i$ is an integer greater than one. We call the expansion in (\ref{1.1}) \textit{backward continued fraction}.
The underlying dynamical system is the R\'enyi map $R$ defined from $[0, 1)$ to $[0, 1)$ by
\begin{equation}\label{1.2}
 R (x) := \frac{1}{1-x} - \left \lfloor \frac{1}{1-x} \right \rfloor,
\end{equation}
which has a neutral fixed point at $0$ and thus is nonuniformly hyperbolic. Here $\lfloor \cdot \rfloor$ denotes the floor function.
R\'enyi showed that the infinite measure $\mathrm{d}x/x$ is invariant for $R$.
This map does not possess a finite absolutely continuous invariant measure, and the usual inducing trick to study its thermodynamic formalism does not work.

Unlike the case (\ref{1.1}), the Gauss map defined from $[0, 1]$ to $[0, 1]$ by
\begin{equation}\label{1.3}
G (x) : = \frac{1}{x} - \left \lfloor \frac{1}{x} \right \rfloor,
\end{equation}
which generates the well-known \textit{regular continued fraction expansion}, possess a finite absolutely continuous invariant measure, namely \textit{Gauss measure} $\mathrm{d}x/{(x+1)}$.
The Gauss map is uniformly expanding, but has infinitely many branches.
The graph of $R$ can be obtained from that of the $G$ by reflecting the latter in the line $x = 1/2$.
It is for this reason that the continued fraction (\ref{1.1}) has been called ``backward".

%
%

Starting from expansion in (\ref{1.1}) and R\'enyi transformation $R$, Gr\"ochenig and Haas \cite{Grochenig&Haas1-1996} define the family of maps
$T_u (x) :=  \frac{1}{u(1-x)} - \lfloor \frac{1}{u(1-x)} \rfloor$, where $u>0$, $x \in [0, 1)$.
Given $u \in (0, 4)$ and $x \in [0, 1)$, $x$ has the $u$-\textit{backward continued fraction expansion}
\begin{equation}
x = 1 - \displaystyle \frac{1}{u n_1 - \displaystyle \frac{1}{n_2 - \displaystyle \frac{1}{u n_3 - \displaystyle \frac{1}{u n_4 \ddots}}}} =:[a_1, a_2, a_3, \ldots]_u, \label{1.4}
\end{equation}
where the integers $n_i=1+a_i \geq 2$ and the coefficient of $n_i$ is $1$ or $u$, depending on the parity of $i$.
In the particular case $u = 1/N$, for positive integers $N \geq 2$, they have identified a finite absolutely continuous invariant measure for $T_u$, namely $\mathrm{d}x / (x+N-1)$.
For $u = 1/N$, where $N$ is a positive integer greater than or equal to $2$, we will call $T_u$ the \textit{R\'enyi-type continued fraction transformation} denoted by $R_N$.

The metrical theory of this algorithm was initiated in \cite{Sebe&Lascu-2018-1}.
The first known metrical problem concerning (regular) continued fractions is due to Gauss.
At the start of the $20$th century an old discovery of Gauss tied the theory of continued fractions to that of probability theory and ergodic theory. In 1802 and 1812 Gauss found the invariant measure of the transformation underlying the regular continued fraction, $G$ in (\ref{1.3}), and asked Lagrange in a letter in 1812 how fast $\lambda \left( G^{-n} ([0, x]) \right)$ converges to the invariant measure
$\gamma ([0, x]) = \log(1+x)/ \log 2$.
Here $\lambda$ is the Lebesgue measure.
In 1928, Kuzmin \cite{Kuzmin-1928} answered to Gauss' question by giving an estimate of the remainder.
Independently in 1929 Paul L\'evy \cite{Levy-1929} improved Kuzmin's result and published another proof.
In the 60's Sz\"usz \cite{Szusz-1961} was able to prove the same result by using Kuzmin's approach.

The purpose of this paper is to prove a Gauss-Kuzmin-L\'evy-type theorem for the R\'enyi-type continued fraction expansions.
In order to solve the problem, we apply the method of Sz\"usz \cite{Szusz-1961}.
A version of Gauss-Kuzmin theorem for these expansions was also studied in \cite{Sebe&Lascu-2018-1} by applying the theory of random systems with complete connections by Iosifescu \cite{IG-2009}.
Namely, using the natural extension for R\'enyi-type continued fraction expansions, we obtained an infinite-order-chain representation of the sequence of the incomplete quotients of these expansions.
Together with the ergodic behaviour of a certain homogeneous random system with complete connections this allowed us to solve a variant of the Gauss-Kuzmin problem.
We mention that applying the Sz\"usz method, we obtain more information on the convergence rate involved.
The main novelty of this paper is the explicit expression in terms of Hurwitz zeta functions of $q_N$ that appears in Theorem \ref{Th.GKL}.
In addition, the estimate we have for $q_N$ shows that $q_N \rightarrow 0$ as $N \rightarrow \infty$.

\section{R\'enyi-type continued fractions}
In this section we briefly present known results about R\'enyi-type continued fractions.

Fix an integer $N \geq 2$. Let the \textit{R\'enyi-type continued fraction transformation} $R_N : [0, 1] \rightarrow [0, 1]$ be given by
\begin{equation} \label{2.1}
R_{N}(x) = \frac{N}{1-x}- \left\lfloor\frac{N}{1-x}\right\rfloor, x \neq 1; \quad R_{N}(1)=0.
\end{equation}
For any irrational $x \in [0, 1]$, $R_N$ generates a new continued fraction expansion of $x$ of the form
\begin{equation} \label{2.2}
x = 1 - \displaystyle \frac{N}{1+a_1 - \displaystyle \frac{N}{1+a_2 - \displaystyle \frac{N}{1+a_3 - \ddots}}} =:[a_1, a_2, a_3, \ldots]_R.
\end{equation}
Here, $a_n$'s are non-negative integers greater than or equal to $N$ defined by
\begin{equation} \label{2.3}
a_1:=a_1(x) = \left\lfloor \frac{N}{1-x} \right\rfloor, x \neq 1; \quad a_1(1)=\infty
\end{equation}
and
\begin{equation}
a_n := a_n(x) = a_1\left( R^{n-1}_N (x) \right), \quad n \geq 2, \label{2.4}
\end{equation}
with $R_{N}^0 (x) = x$.

The rational approximants to $x$ arise in a manner similar to that in the case of other continued fraction algorithms.
In particular we define two integer sequences by
$p_0=1$, $q_0=1$, $p_1=1+a_1-N$, $q_1=1+a_1$,
\begin{equation*}
p_n = (1+a_n) p_{n-1} - N p_{n-2} \mbox{ and } q_n = (1+a_n) q_{n-1} - N q_{n-2}
\end{equation*}
for $n \geq 2$.  A simple inductive argument gives
\begin{equation*}
p_{n-1}q_n - p_nq_{n-1} = N^n, \quad n \in \mathbb{N_+}:=\{1, 2, \ldots \},
\end{equation*}
and whence $p_n$ and $q_n$ are coprime.
The sequence of rationals $\left\{{p_n}/{q_n}\right\}$, ${n \in \mathbb{N_+}}$ are the convergents to $x$ in $[0, 1]$.
In \cite{Grochenig&Haas1-1996} it was shown that the dynamical system $\left([0,1],{\mathcal B}_{[0,1]}, R_N, \rho_N \right)$ is measure preserving and ergodic. Here, $\mathcal{B}_{[0,1]}$ denotes the $\sigma$-algebra of all Borel subsets of $[0,1]$,
and the probability measure $\rho_N$ is defined by
\begin{equation}
\rho_N (A) :=
\frac{1}{\log \left(\frac{N}{N-1}\right)} \int_{A} \frac{\mathrm{d}x}{x+N-1}, \quad A \in {\mathcal{B}}_{[0,1]}. \label{2.5}
\end{equation}

In \cite{Sebe&Lascu-2018-1} we investigated the Perron-Frobenius operator of $R_N$
on the measurable space $\left([0, 1], {\mathcal{B}}_{[0,1]}, \mu \right)$ such that the probability measure $\mu$ satisfies
$\mu\left(R_N^{-1}(A)\right) = 0$ whenever $\mu(A) = 0$ for $A \in {\mathcal{B}}_{[0, 1]}$.
Especially, we studied the Perron-Frobenius operator $U$ of $\left([0, 1], {\mathcal{B}}_{[0,1]}, \rho_N, R_N \right)$,
that is, $U$ is a unique operator on
$L^1([0,1], \rho_N):=\{f: [0, 1] \rightarrow \mathbb{C} : \int^{1}_{0} |f |\mathrm{d}\rho_N < \infty \}$ which satisfies
\begin{equation}
\int_{A} U f \,\mathrm{d}\rho_N = \int_{R_{N}^{-1}(A)} f\, \mathrm{d}\rho_N \quad \mbox{ for any }
A \in {\mathcal{B}}_{[0, 1]},\, f \in L^1 \left([0, 1], \rho_N \right). \label{2.6}
\end{equation}
Also, we have found an explicit formula for the Perron-Frobenius operator under the invariant measure $\rho_N$, namely
\begin{equation}
Uf(x) = \sum_{i \geq N} P_{N,i}(x)\,f\left(u_{N,i}(x)\right), \quad f \in L^1([0, 1], \rho_{N}), \label{2.7}
\end{equation}
where $P_{N,i}$ and $u_{N,i}$ are functions defined on $[0, 1]$ by:
\begin{equation}\label{2.8}
P_{N,i}(x) := \frac{x+N-1}{(x+i)\,(x+i-1)}
\end{equation}
and
\begin{equation}\label{2.9}
u_{N,i}(x) := 1 - \frac{N}{x+i}.
\end{equation}

A more thorough account of R\'enyi-type continued fractions can be found in \cite{Grochenig&Haas1-1996, Sebe&Lascu-2018-1, Sebe&Lascu-2018-2}.

\section{Main result}
In this section we show our main theorem.
Let $\mu$ be a non-atomic probability measure on $\mathcal{B}_{[0,1]}$ and define
\begin{eqnarray}
F_{N,0} (x) &:=& \mu ([0,x]), \ x \in [0, 1], \\ \label{3.1}
F_{N,n} (x) &:=& \mu (R_{N}^n \leq x), \ x \in [0, 1], \ n \in \mathbb{N}_+. \label{3.2}
\end{eqnarray}
Then the following holds.

\begin{theorem}(A Gauss-Kuzmin-L\'evy-type theorem) \label{Th.GKL}
Let $R_N$ and $F_{N,n}$ be as in $(\ref{2.1})$ and $(\ref{3.2})$.
Then there exists a constant $0 < q_N < 1$ such that $F_n$ can be written as
\begin{equation}
F_{N,n} (x) = \frac{1}{\log \left(\frac{N}{N-1}\right)} \log \left(\frac{x+N-1}{N-1}\right) + \mathcal{O}(q_N^n) \label{3.3}
\end{equation}
uniformly with respect to $x \in [0, 1]$.
\end{theorem}
\begin{remark}
From $(\ref{3.3})$, we see that
\begin{equation}
\lim_{n\rightarrow\infty} F_{N,n}(x) = \rho_N ([0,x]), \label{3.4}
\end{equation}
where $\rho_N$ is the measure defined in $(\ref{2.5})$.
In fact, Theorem $\ref{Th.GKL}$ estimates the error
\begin{equation}
e_{N,n}(x) = e_{N,n}(x, \mu) = \mu (R_N^n \leq x) - \rho_N ([0,x]), \quad x \in [0, 1]. \label{3.5}
\end{equation}
\end{remark}
To prove Theorem \ref{Th.GKL} we need the following results.
\begin{lemma} \label{G-K.eq.}
For functions $\{F_{N,n}\}$ in $(\ref{3.2})$, the following Gauss-Kuzmin-type equation holds:
\begin{equation}
F_{N,n+1} (x) = \sum_{i \geq N}\left\{F_{N,n}\left(1 - \frac{N}{x+i}\right) - F_{N,n}\left(1-\frac{N}{i}\right)\right\} \label{3.6}
\end{equation}
for $x \in [0,1]$ and $n \in \mathbb{N}$.
\end{lemma}
\begin{proof}
From (\ref{2.1}) and (\ref{2.4}), we see that
\begin{equation}\label{3.7}
R_N^n(x) = 1 - \frac{N}{a_{n+1}+R_N^{n+1}}\, , \quad n \in \mathbb{N}_+.
\end{equation}
Now,
\begin{eqnarray*}
F_{N,n+1} (x) &=& \mu \left(R_{N}^{n+1} \leq x \right) = \sum_{i \geq N} \mu \left( 1-\frac{N}{i} \leq R_{N}^{n+1} \leq 1-\frac{N}{i+x} \right) \\
              &=& \sum_{i \geq N} \left\{F_{N,n}\left(1 - \frac{N}{x+i}\right) - F_{N,n}\left(1-\frac{N}{i}\right)\right\}.
\end{eqnarray*}
\end{proof}
\begin{remark}
Assume that for some $p \in\mathbb{N}$, the derivative $F'_{N,p}$ exists everywhere in $[0, 1]$ and is bounded.
Then it is easy to see by induction that $F'_{N,p+n}$ exists and is bounded for all $n \in \mathbb{N}_+$.
This allows us to differentiate $(\ref{3.6})$ term by term, obtaining
\begin{equation}
F'_{N,n+1}(x) = \sum_{i \in N} \frac{N}{(x+i)^2}F'_{N,n}\left(1-\frac{N}{x+i}\right). \label{3.8}
\end{equation}
\end{remark}
We introduce functions $\{ f_{N,n} \}$ as follows:
\begin{equation}
f_{N,n}(x) := (x+N-1)F'_{N,n}(x), \quad x \in [0, 1], \ n \in \mathbb{N}. \label{3.9}
\end{equation}
Then (\ref{3.8}) is
\begin{equation}
f_{N,n+1}(x) = \sum_{i \geq N} P_{N,i}(x)f_{N,n}\left(u_{N,i}(x)\right), \label{3.10}
\end{equation}
where $P_{N,i}(x)$ and $u_{N,i}(x)$ are given in (\ref{2.8}) and (\ref{2.9}), respectively.
\begin{lemma} \label{lem.3.2}
For $\{f_{N,n}\}$ in $(\ref{3.9})$, define $M_{N,n} := \displaystyle \max_{x \in [0, 1]}|f'_{N,n}(x)|$.
Then
\begin{equation}
M_{N,n+1} \leq q_N \cdot M_{N,n} \label{3.11}
\end{equation}
where
\begin{equation}
q_N = \sum_{i \geq N} \left( \frac{1}{i^3} +\frac{N}{i^2(i+1)} \right). \label{3.0120}
\end{equation}
\end{lemma}
\begin{proof}
Since
\begin{equation*}
P_{N,i}(x) =  \frac{i+1-N}{x+i} - \frac{i-N}{x+i-1},
\end{equation*}
we have
\begin{eqnarray} \label{3.012}
f'_{N,n+1}(x) &=& \sum_{i \geq N}
\left\{
P'_{N,i}(x) f_{N,n}\left(u_{N,i}(x)\right) + P_{N,i}(x) f'_{N,n}\left(u_{N,i}(x)\right) u'_{N,i}(x)
\right\} \nonumber \\
&=& \sum_{i \geq N}
\left\{
\left( \frac{i-N}{(x+i-1)^2} - \frac{i+1-N}{(x+i)^2} \right) f_{N,n}\left(u_{N,i}(x)\right) + P_{N,i}(x) f'_{N,n}\left(u_{N,i}(x)\right) \frac{N}{(x+i)^2}
\right\} \nonumber \\
&=& \sum_{i \geq N}
\left\{
\frac{i+1-N}{(x+i)^2}\left[ f_{N,n}\left(u_{N,i+1}(x)\right) - f_{N,n}\left(u_{N,i}(x)\right) \right]+P_{N,i}(x) f'_{N,n}\left(u_{N,i}(x)\right) \frac{N}{(x+i)^2}
\right\} \nonumber \\
&=& \sum_{i \geq N}
\left\{\frac{i+1-N}{(x+i)^3(x+i+1)} f'_{N,n}\left(\theta_i\right) + f'_{N,n}\left(u_{N,i}(x)\right) \frac{NP_{N,i}(x)}{(x+i)^2}
\right\}
\end{eqnarray}
where $u_{N,i+1}(x) < \theta_i < u_{N,i}(x)$. Now (\ref{3.012}) implies
\begin{equation}
M_{N,n+1} \leq M_{N,n} \cdot \max_{x \in [0, 1]} \left( \sum_{i \geq N} \frac{i+1-N}{(x+i)^3(x +i+1)}
+ N \sum_{i \geq N} \frac{x+N-1}{(x+i)^3(x+i-1)} \right). \label{3.013}
\end{equation}
We now must calculate the maximum value of the sums in this expression.
Using that $x \in [0, 1]$ and $i \geq N$, we get
\begin{equation*}
\frac{i+1-N}{(x+i)^3(x +i+1)} \leq \frac{i+1-N}{i^3(i+1)}
\end{equation*}
and
\begin{equation*}
\frac{x+N-1}{(x+i)^3(x+i-1)} \leq \frac{1}{(x+i)^3} \leq \frac{1}{i^3}.
\end{equation*}
Thus,
\begin{equation}
M_{N,n+1} \leq M_{N,n} \cdot \sum_{i \geq N}\left( \frac{1}{i^3} +\frac{N}{i^2(i+1)} \right) \label{3.016}
\end{equation}
and the proof is complete.
\end{proof}
\noindent \textbf{Proof of Theorem \ref{Th.GKL}}.
For $\{F_{N,n}\}$ in $(\ref{3.2})$, we introduce a function $W_{N,n}(x)$ such that
\begin{equation}
F_{N,n} (x) = \frac{1}{\log \left(\frac{N}{N-1}\right)} \log \left(\frac{x+N-1}{N-1}\right) + W_{N,n}(x) \label{3.12}
\end{equation}

Because $F_{N,n}(0)=0$ and $F_{N,n}(1)=1$, we have $W_{N,n}(0)=W_{N,n}(1)=0$.
To prove Theorem \ref{Th.GKL}, we have to show the existence of a constant $0 < q_N < 1$ such that
\begin{equation}
W_{N,n}(x) = {\mathcal O}(q^n). \label{3.13}
\end{equation}

For $\{f_{N,n}\}$ in (\ref{3.9}), if we can show that
$f_{N,n}(x) = \frac{1}{\log \left(\frac{N}{N-1}\right)} + {\mathcal O}(q_N^n)$, then integrating (\ref{3.9}) will show (\ref{3.3}).

To demonstrate that $f_{N,n}(x)$ has this desired form, it suffices to prove the following lemma.
\begin{lemma} \label{lem.3.6}
For any $x \in [0, 1]$ and $n \in \mathbb{N}$ there exists a constant $q_N:=q_N(x)$ with $0 < q_N < 1$ such that
\begin{equation}
f'_{N,n}(x) = {\mathcal O}(q_N^n). \label{3.14}
\end{equation}
Moreover, for any positive integer $N \geq 2$ the following estimate holds
\begin{equation}\label{3.19}
  \frac{1}{N^3} + \frac{1}{2N(N+1)} + \frac{1}{2N} < q_N < \frac{1}{2N( N - 1)} +\frac{1}{N} - \frac{1}{2N+1}.
\end{equation}
\end{lemma}

\begin{proof}
Let $q_N$ be as in Lemma \ref{lem.3.2}. Using this lemma, to show (\ref{3.14}) it is enough to prove that $q_N < 1$.
%
First, we will write $q_N$ in terms of Hurwitz zeta functions. Thus,
\begin{eqnarray*}
q_N &=& \sum_{i \geq N} \left( \frac{1}{i^3} +\frac{N}{i^2(i+1)} \right) = \sum_{i \geq N} \left( \frac{1}{i^3}+\frac{N}{i^2}\right) - 1 \\
    &=& \zeta(3, N)+N \zeta(2, N)-1.
\end{eqnarray*}
For $i \geq N$ and $a:=\frac{1}{2}\left( \sqrt{4N^2+1} -(2N+1) \right) > -\frac{1}{2}$, we have
\[
a^2+(2N+1)a+N=0
\]
and
\[
a^2+(2N+1)a+N+(2a+1)(i-N) \geq 0,
\]
i.e.,
\[
i^2 \leq (i+a)(i+1+a).
\]
Hence,
\[
\zeta(2,N) \geq \sum_{i\geq N} \left( \frac{1}{i+a} - \frac{1}{i+1+a} \right) = \frac{1}{N+a} = \frac{2}{\sqrt{4N^2+1}-1}.
\]
Also we have
\[
\zeta(2,N) < \frac{1}{N^2} + \sum_{i \geq N+1} \frac{1}{(i-1/2)(i+1/2)} = \frac{1}{N^2} +\frac{2}{2N+1}.
\]
For $i \geq N$ and $b:=N\left(\sqrt{N^2+1}-N\right)<\frac{1}{2}$, we have
\[
b^2+2N^2b-N^2=0
\]
and
\[
b^2+2N^2b-N^2+(2b-1)(i^2-N^2) \leq 0,
\]
i.e.,
\[
i^4\geq (i^2-i+b)(i^2+i+b).
\]
Hence
\begin{eqnarray*}
\zeta(3,N) &<&\frac{1}{2}\sum_{i \geq N} \left( \frac{1}{(i-1)i+b} - \frac{1}{i(i+1)+b} \right) = \frac{1}{2(N^2-N+b)} \\
           &=& \frac{1}{2N(\sqrt{N^2+1} - 1)}.
\end{eqnarray*}
Also, we have
\begin{eqnarray*}
\zeta(3,N) &>& \frac{1}{N^3}+\frac{1}{2} \sum_{i \geq N+1} \left( \frac{1}{i^2-i+1/2} + \frac{1}{i^2+i+1/2} \right) \\
           &=& \frac{1}{N^3} + \frac{1}{2(N^2+N+1/2)}.
\end{eqnarray*}
Therefore,
\begin{eqnarray*}
q_N &<& \frac{1}{2N(\sqrt{N^2+1} - 1)} + N \left( \frac{1}{N^2} +\frac{2}{2N+1} \right)-1 \\
    &<& \frac{1}{2N(N-1)} + \frac{1}{N} - \frac{1}{2N+1},
\end{eqnarray*}
and
\begin{eqnarray*}
q_N &>& \frac{1}{N^3} + \frac{1}{2(N^2+N+1/2)} + \frac{2N}{\sqrt{4N^2+1}-1} - 1 \\
    &>& \frac{1}{N^3} + \frac{1}{2(N^2+N+1/2)} + \frac{2N+1}{2N} - 1 \\
    &=& \frac{1}{N^3} + \frac{1}{2(N^2+N+1/2)} + \frac{1}{2N} \\
    &>&  \frac{1}{N^3} + \frac{1}{2N(N+1)} + \frac{1}{2N}.
\end{eqnarray*}

For example, we have
\begin{center}

\begin{tabular}{| l | l | l | c | c | c |}
\hline
N               & Lower bound of $q_N$                 &  Upper bound of $q_N$  \\ \hline\hline
2               & $0.4583333333333333$                 &  $0.55$              \\  \hline
10              & $0.055545454545454544$               &  $0.05793650793650794$ \\ \hline
100             & $0.00505050495049505$                &  $0.0050753806723955975$  \\ \hline
500             & $0.001002004007984032$               &  $0.001003003009015033$ \\ \hline
1000            & $0.0005005005004995005$              &  $0.0005007503755629693$ \\ \hline
5000            & $0.00010002000400079984$             &  $0.00010003000300090015$ \\ \hline
10000           & $0.00005000500050004999$             &  $0.000050007500375056254$ \\ \hline
\end{tabular}
\end{center}

\end{proof}

\end{document}